\documentclass[twoside,12pt, leqno]{article}

\usepackage{etex}
\usepackage{amscd,amsxtra,latexsym,epsfig,epic,graphics}
\usepackage[matrix,arrow,curve]{xy}
\usepackage{graphicx}
\usepackage[bookmarksnumbered]{hyperref}
\hypersetup{colorlinks=true,backref=true,citecolor=blue}
\usepackage{array,ragged2e}
\usepackage[top=28mm,bottom=28mm]{geometry}
\usepackage[english]{babel}
\usepackage[utf8]{inputenc}
\usepackage{times}
\usepackage{amsmath,amssymb,amsfonts,amsthm}
\usepackage{setspace}
\usepackage{color}
\usepackage{paralist}
\usepackage{pgfplots}
\usepackage{mathrsfs}
\usepackage{mathtools}
\usepackage{bm}
\usepackage{tikz-cd}
\usepackage{tikz}
\usetikzlibrary{chains,
	matrix,
	positioning,
	scopes,arrows
}
\usepackage{pdfpages}
\usepackage{bibentry}
\usepackage{multirow}

\newtheorem{theorem}{Theorem}[section]
\newtheorem{lemma}[theorem]{Lemma}
\newtheorem{proposition}[theorem]{Proposition}
\newtheorem{corollary}[theorem]{Corollary}
\theoremstyle{definition}
\newtheorem{definition}[theorem]{Definition}
\newtheorem{example}[theorem]{Example}
\newtheorem{notation}[theorem]{Notation}
\theoremstyle{remark}
\newtheorem{remark}[theorem]{Remark}

\newcommand{\sO}{{\cal O}}

\renewcommand{\P}{{\mathbb P}}

\DeclareMathOperator{\im}{im}
\DeclareMathOperator{\depth}{depth}

\DeclareMathOperator{\coker}{coker}
\DeclareMathOperator{\Ext}{Ext}

\DeclareMathOperator{\ann}{ann}
\DeclareMathOperator{\Proj}{Proj}

\DeclareMathOperator{\Hom}{Hom}

\DeclareMathOperator{\id}{id}

\def\ZZ{{\mathbb Z}}
\def\QQ{{\mathbb Q}}
\def\CC{{\mathbb C}}

\def\PP{{\mathbb P}}

\def\PP{{\mathbb P}}
\def\PP{{\mathbb P}}

\def\CO{{\mathcal O}}

\def\Mac2{{\emph{Macaulay2}}}

\def\canmod{X_{can}}
\DeclareMathOperator{\Tors}{Tors}

\makeatletter
\def\Ddots{\mathinner{\mkern1mu\raise\p@
\vbox{\kern7\p@\hbox{.}}\mkern2mu
\raise4\p@\hbox{.}\mkern2mu\raise7\p@\hbox{.}\mkern1mu}}
\makeatother

\DeclareMathOperator{\Aut}{Aut}
\DeclareMathOperator{\PGL}{PGL}

\DeclareMathOperator{\St}{St}
\makeatletter
\def\Ddots{\mathinner{\mkern1mu\raise\p@
\vbox{\kern7\p@\hbox{.}}\mkern2mu
\raise4\p@\hbox{.}\mkern2mu\raise7\p@\hbox{.}\mkern1mu}}
\makeatother

\usepackage{times}
\newdimen\x \x=12pt

\usepackage{color}

\date{}
\title{An 8-dimensional family of simply connected Godeaux surfaces}
\author{Frank-Olaf Schreyer and Isabel Stenger}

\begin{document}

\maketitle
\begin{abstract}

In this paper we describe a construction method for numerical Godeaux surfaces based on homological algebra.  We show the existence of an 8-dimensional locally complete family of simply connected numerical Godeaux surfaces. 
\end{abstract}
\tableofcontents
\section*{Introduction}
The minimal surfaces of general type with the smallest possible numerical invariants $p_g = q = 0$, $K^2 = 1$ are the numerical Godeaux surfaces. They have always been of a particular interest in the classification of algebraic surfaces. Miyaoka (\cite{Miyaoka}) showed that the torsion group of such surfaces is cyclic of order $m \leq 5$. Whereas numerical Godeaux surfaces for $m=3,4,5$ are completely described by Reid (\cite{ReidGodeaux78}), a complete classification in the other cases is still open.

Let $X$ be a numerical Godeaux surface, and let
$x_0,x_1$ (respectively $y_0,\ldots,y_3$) denote a basis of $H^0(X,\CO_X(2K_X))$ (respectively of $H^0(X,\CO_X(3K_X))$). We consider the weighted polynomial ring $S = \Bbbk[x_0,x_1,y_0,\ldots,y_3]$. Then the canonical ring $R(X)$ is a finitely generated $S$-module (see Lemma \ref{lem_canringfinggenmod}). 
Using the structure result from \cite{Stenger18}, $R(X)$ as an $S$-module admits a self-dual minimal free resolution $\textbf{F}$ of length three with a skew-symmetric middle map. 

In Section \ref{structure thm} we determine a set of algebra generators for the canonical ring $R(X)$ of $X$ and show that there is an embedding of the canonical model in a weighted projective space $$\canmod \hookrightarrow \P(2^2,3^4,4^4,5^3).$$ Now, considering $R(X)$ as an $S$-module corresponds to studying the canonical model $\canmod$ via its projection to $\Proj(S) = \P(2^2,3^4)$. In Lemma \ref{lem_normalization} we show that the projection map is birational to its image $Y \subset \P(2^2,3^4)$.

In this article, we study the case where
the bicanonical system $|2K_X|$ has no fixed part and four distinct base points, which we can assume to be mapped to the coordinate points of $\PP^3$ under the tricanonical map. Enumerating these base points, we introduce the notion of a \textit{marked numerical Godeaux surface}. Blowing up the four base points, the bicanonical system induces a fibration to $\PP^1$.

In Section \ref{normal form} we present our construction method. The main idea is to recover the minimal free resolution $\textbf{F}$ of $R(X)$ as a deformation of the complex $\textbf{F}/(x_0,x_1)$. In Proposition \ref{descriptionFmod} we give a complete description of $\textbf{F}/(x_0,x_1)$ for a  marked numerical Godeaux surface. 
The essential technique is then to consider an unfolding of $\textbf{F}/(x_0,x_1)$ and to interpret the flatness as equations with respect to the unfolding parameters.
After removing unfolding parameters which have to be zero by the equations, the parameters which are linear functions in $x_0,x_1$ satisfy a quadratic system of equations. This system consists of four quadrics defining a complete intersection variety $Q$ in a $\P^{11}$. The first step in our construction method 
for numerical Godeaux surfaces is to choose a line $\ell$ in $Q$. After this choice, the remaining relations define a linear system of equations (depending on $\ell$) which can be solved by a syzygy computation. 

In Section \ref{Fano} we analyze the Fano variety of lines $F_1(Q)$ and deduce a method to compute points in $F_1(Q)$. We show that $F_1(Q)$ is birational to a subvariety of four copies of  $\P^3$ but fail to give a complete birational classification. Moreover, we study the induced  action of the  group fixing the coordinate points in $\P^3$ on $F_1(Q)$. 

In Section \ref{subsec C-complex} we present the construction of an 8-dimensional locally complete family of simply connected numerical Godeaux surfaces.  We call this family the \textit{dominant component} as its corresponding lines dominate  $F_1(Q)$. We introduce special loci in $Q$ arising from our construction and call a line in $Q$ \textit{general} if it does not intersect any of these loci. We show in Theorem \ref{dominant component} that for a general line in $Q$ we get a 4-dimensional linear solution space in the second construction step. Thus, we obtain a $\P^3$-bundle over an open subset of $F_1(Q)$ which is 8-dimensional. Taking the group action into account, we get in total a family of dimension 8. 
Moreover, we identify the simply connected Barlow surfaces (see \cite{Barlow})  within our family.

In a further work, we show that the torsion $\ZZ/3\ZZ$- and $\ZZ/5\ZZ$-components of Reid (\cite{ReidGodeaux78}) correspond to lines meeting special loci in $Q$ and we give an alternative proof of the unirationality of these components using our method. 
Furthermore, we give a complete characterization for the existence of hyperelliptic bicanonical fibers and non-trivial torsion groups in terms of our homological setting. 

{\bf Acknowledgments.}  This work was funded by the Deutsche Forschungsgemeinschaft, SfB-TRR 195. We thank Wolfram Decker and Miles Reid for inspiring conversations and Thomas Breuer for his helpful comments on the group action. Our work makes essential use of $\Mac2$ (\cite{M2}). We thank Dan Grayson and Mike Stillman for their program. Our work would not have been possible without computer algebra.

\section{Preliminaries}
Throughout this paper, we use the following notation. 
\begin{itemize}
	\item $X$ denotes a numerical Godeaux surface;
	\item $\pi\colon X \rightarrow \canmod = \Proj(R(X))$ denotes the morphism to the canonical model;
	\item $K_X$ and $K_{\canmod}$ denote canonical divisors;
	\item $\Tors X$  denotes the torsion subgroup of the Picard group;
	\item $\Bbbk$ denotes the ground field.
\end{itemize}
We are mainly interested in the case $\Bbbk = \CC$ but for computations we also use $\Bbbk = \QQ$ or number fields. In our experiments, $\Bbbk$ can also be a finite field which we often may regard as a specialization of a number field.

In our construction we use some classical results on the bi- and the tricanonical system of a numerical Godeaux surface  $X$ over $\CC$ which we will briefly recall here. 
Let us start with the bicanonical system. We write 
$$ |2K_X| = |M| + F,$$ where $M$ denotes a generic member of the moving part and $F$ the fixed part of $|2K_X|$. 
\begin{proposition}[\cite{Miyaoka}, Lemma 6]\label{prop_bicanoncond}
	If $M$ is generically chosen, $M$ is reduced and irreducible. Moreover, $M$ and $F$ satisfy one of the following conditions
	\begin{enumerate}[(i)]
		\item $F = 0$,
		\item $K_XF = 0, F^2 = -2, M^2 = 2, MF = 2$,
		\item $K_XF = 0, F^2 = -4, M^2 = 0, MF = 4$.
	\end{enumerate}
\end{proposition}
\begin{remark}\label{rem_bicanonsys}
	The statement shows that the fixed part of $|2K_X|$, if non-empty, is supported on the $(-2)$-curves of $X$. Hence,
	$|2K_{\canmod}|$ is free from fixed components and its generic member is irreducible.
\end{remark}

Next we summarize some results on the tricanonical system:
\begin{theorem}[\cite{Miyaoka}, Theorem 3, Proposition 2 and Proposition 3]\label{prop_tricanbicanbase}
	The tricanonical map $\phi_{|3K_{X}|}$ is birational onto its image. 
	The linear system $|3K_X|$ has no fixed part. The generic element $M$ of the moving part of $|2K_X|$ contains no base points of $|3K_X|$.
\end{theorem}
\begin{remark}\label{rem_emptybaseloc}
	Combining the previous statements, we conclude that no base point of $|3K_X|$ (respectively of $|3K_{\canmod}|$) is a base point of $|2K_X|$ (respectively of $|2K_{\canmod}|$). Hence, for a base point $P$ of $|3K_X|$ there exists a unique divisor $D \in |2K_X|$ which contains $P$. 
	Furthermore, Miyaoka showed that a point $\hat{P}$ is a base point of $|3K_{\canmod}|$ if and only if $\hat{P} = \hat{D_1}\hat{D_2}$, where $\hat{D_1},\hat{D_2}$ are two distinct effective curves which are numerically equivalent to $K_{\canmod}$ with $\hat{D_1} + \hat{D_2} \in |2K_{\canmod}|$. The last fact gives indeed a very precise description of the number of base points of $|3K_X|$.
\end{remark}
\begin{theorem}[\cite{Miyaoka}, Theorem 2]
	Every base point of the tricanonical system $|3K_X|$ is simple, and the number $b$ of base points is given as follows:
	\[ b =  \frac{ \# \{ t \in H^2(X,\ZZ)_\textup{tors} \mid t \neq -t \}}{2} 
	.\]
\end{theorem}
Note that for  a numerical Godeaux surface $X$,  $H^2(X,\ZZ)_\textup{tors} = \Tors X = H_1(X,\ZZ)$. 
Bombieri showed that the order of the torsion group of a numerical Godeaux surface is $\leq 6$ (see \cite{Bombieri}, Theorem 11.14).
Miyaoka refined this result in the following way: 
\begin{lemma}[\cite{Miyaoka}, Lemma 11]
	The torsion group of $X$ is cyclic of order $ \leq 5$.
\end{lemma}
Combining these two statements we obtain the following important result:
\begin{theorem}[\cite{Miyaoka}, Theorem 2' and subsequent Remark]{\label{cor_tricanbase}}
	As above, let $b$ denote the number of base points of $|3K_X|$. Then
	$$ b = \begin{cases} 
	0 &\mbox{if } \Tors X \cong 0 \text{ or } \ZZ/2\ZZ,  \\
	1 &\mbox{if } \Tors X \cong \ZZ/3\ZZ \text{ or } \ZZ/4\ZZ, \\
	2 &\mbox{if } \Tors X \cong \ZZ/5\ZZ. \end{cases}$$
\end{theorem}
Later we will use this characterization and Lemma \ref{lem_oddtorsion} below to determine the torsion group of our constructed surfaces. 
\begin{remark}\label{rem_h0tors}
	Note that for a non-trivial torsion element $\tau \in \Tors X$ we have 
	\[h^0(X,K_X+\tau) = 1, \ h^1(X,K_X + \tau) = 0.\] 
	Indeed, by the Riemann-Roch theorem we have $h^0(X,K_X+\tau) - h^1(X,K_X+\tau) = 1$. Now $h^1(X,K_X+\tau) > 0$ implies that for the finite \'{e}tale covering $f\colon Y \rightarrow X$ corresponding to $\tau$ we get $h^1(Y,{\cal O}_Y) > 0$. Hence, $Y$ has finite cyclic coverings of any large order (see \cite{Bombieri}, Lemma 10.14), and so does $X$, which is not possible.
\end{remark}

\begin{lemma}\label{lem_oddtorsion}
	Assume that $|2K_X|$ has no fixed part and $(2K_X)^2 = 4$  distinct (simple) base points. Then the order of $\Tors X$ is odd.  In particular, for the number $b$ of base points of $|3K_X|$ we find that
	\begin{itemize}
		\item $b = 0$ if and only if $\Tors X \cong 0$,
		\item $b = 1$ if and only if $\Tors X \cong \ZZ/3\ZZ$,
		\item $b = 2$ if and only if $\Tors X \cong \ZZ/5\ZZ$.
	\end{itemize}
	
\end{lemma}
\begin{proof} 
	Suppose to the contrary that $\Tors X \cong \ZZ/2\ZZ$ or $\Tors X \cong \ZZ/4\ZZ$. Let $\tau \in \Tors X$ be a non-trivial torsion element of order 2. By Remark \ref{rem_h0tors} there exists an effective divisor $D \in |K_X + \tau|$. 
	But then $|2K_X|$ contains the double curve $2D$ and thus, cannot have 4 distinct base points. The second part is an immediate consequence of Theorem
	\ref{cor_tricanbase}.
\end{proof}

\section{A structure theorem for Godeaux rings}\label{structure thm}
In this section we present a structure theorem for the canonical ring 
\[ R(X) = \bigoplus_{n} H^0(X,\sO_X(nK_X)) = \bigoplus_{n} H^0(X,nK_X)  \]
of a numerical Godeaux surface $X$ in which we describe the minimal free resolution of $R(X)$ as a module over a weighted polynomial ring $S$. 

We first determine a minimal set of generators of $R(X)$ as an $\Bbbk$-algebra. 
Using the Riemann-Roch theorem we see that the plurigenera of $X$ are 
$$ P_n = h^0(X,nK_X) = 
\begin{cases}
1 & \text{for } n = 0, \\
0 & \text{for } n = 1, \\
\binom{n}{2} + 1 & \text{for } n \geq 2.
\end{cases} $$
Let $x_0,x_1$ be a basis of $H^0(X,2K_X)$, and let $y_0,y_1,y_2,y_3$ be a basis of $H^0(X,3K_X)$.
Now $R(X)$ being an integral domain implies that the elements $x_0^2,x_0x_1,x_1^2$ are linearly independent. Thus, 
as $H^0(X,4K_X)$ is $7$-dimensional, we can choose $z_0,\ldots,z_3 \in H^0(X,4K_X)$ extending these elements to a basis. To give a basis for the vector space $H^0(X,5K_X)$, we use the following: 
\begin{lemma}
	The multiplication map $\mu\colon H^0(X,2K_X) \otimes H^0(X,3K_X) \rightarrow H^0(X,5K_X)$ is injective.
\end{lemma}
\begin{proof} 
	As $R(X) \cong R(\canmod)$ it is sufficient to show that \[\widetilde{\mu}\colon H^0(\canmod,2K_{\canmod}) \otimes H^0(\canmod,3K_{\canmod}) \rightarrow H^0(\canmod,5K_{\canmod})\] is injective. 
	Let ${x_0},{x_1}$ be a basis of $H^0(X,2K_{X}) \cong H^0(\canmod,2K_{\canmod})$. 
	By Remark \ref{rem_bicanonsys} we know that the bicanonical system has no fixed part on the canonical model. Hence, the following sequence is exact
	\[ 0 \rightarrow \sO_{\canmod}(-4K_{\canmod}) 
	\xrightarrow{\begin{footnotesize}
		\begin{pmatrix}
		{x_1} \\ -{x_0}
		\end{pmatrix} \end{footnotesize}} 
	\begin{matrix}
	\sO_{\canmod}(-2K_{\canmod}) \\
	\oplus \\
	\sO_{\canmod}(-2K_{\canmod}) \\
	\end{matrix}
	\xrightarrow{({x_0},{x_1})} 
	\sO_{\canmod} \rightarrow \sO_{Z} \rightarrow 0,\]
	where $Z = \operatorname{div}({x_0}) \cap \operatorname{div}({x_1})$.
	Now tensoring with $\sO_{\canmod}(5K_{\canmod})$ and taking global sections, the statement follows since $h^0(\canmod,K_{\canmod}) = h^1(\canmod,\sO_{\canmod}) = 0$.  
\end{proof}
The lemma shows that the global sections $x_iy_j$ for $i = 0,1$ and $ j = 0,\ldots,3$ define an $8$-dimensional subspace of $H^0(X,5K_X)$. Now as $h^0(X,5K_X) = 11$, we can choose sections $w_0,w_1,w_2 \in H^0(X,5K_X)$ extending these elements to a basis. 
Since we will use the same notation for the generators in the following, we summarize the previous results in one table:
\begin{center}
	\begin{tabular}{l|c|l}
		$n$ & $h^0(X,nK_X)$ & basis of $H^0(X,nK_X)$ \\ \hline
		2 & 2 & \color{blue}$x_0,x_1$ \\
		3 & 4 & \color{blue} $y_0, \ldots, y_3$ \\
		4 & 7 & $x_0^2,x_0x_1,x_1^2,\color{blue}z_0,\ldots,z_3$ \\
		5 & 11 & $x_0y_0,\ldots,x_1y_3,\color{blue}w_0,w_1,w_2$
	\end{tabular}
\end{center}
The entries marked in blue give a minimal generating set of $R(X)$ (as a $\Bbbk$-algebra) up to degree 5. 
Ciliberto showed that the canonical ring of any surface of general type is generated in degree $\leq 6$ (see \cite{Ciliberto}, Theorem 3.5). Using this result, we get the following refinement for numerical Godeaux surfaces:

\begin{lemma}\label{prop_canonicalring5}
	As a $\Bbbk$-algebra, $R(X)$ is generated in degree $\leq 5$.  
\end{lemma}
\begin{proof}
	First recall that no base point of the tricanonical system of $X$ is a base point of the bicanonical system. Hence, there exists a global section $y \in H^0(X,3K_X)$ such that $\Proj(R(X)/(x_0,x_1,y) = \emptyset$.   
	Then, as $R(X)$ is Cohen-Macaulay, $R(X)$ is a free module over $A = \Bbbk[x_0,x_1,y]$. Using the Hilbert series $\Psi(t)$ of $R(X)$,  the degrees of the free generators are given by
	$$ (1-t^2)^2(1-t^3)\Psi(t)  = 1 + 3t^3 + 4t^4+3t^5+t^8.$$
	This implies that any homogeneous element of $R(X)$ of degree 6 is an $A$-linear combination of elements of $R(X)$ of degree $\leq 5$ which shows the claim. 
\end{proof}
So, if $\hat{S} := \Bbbk[x_0,x_1,y_0,\ldots,y_3,z_0,\ldots,z_3,w_0,w_1,w_2]$ is the weighted polynomial ring with degrees as assigned before, there is a closed embedding 
\begin{equation}\label{equ_closedemb}
\canmod =  \Proj(R(X)) \xhookrightarrow{\quad } \Proj(\hat{S}) = \P(2^2,3^4,4^4,5^3).
\end{equation}
Thus, we can consider the canonical model of $X$ as a subvariety of a weighted projective space of dimension 12.
Studying this embedding is difficult because we have no structure theorem for Gorenstein ideals of such high codimension. Furthermore, from a computational point of view, codimension 10 is not promising for irreducibility or non-singularity tests. The original construction idea in \cite{Schreyer05} addresses these problems:

\textbf{Basic idea}: We do not consider $R(X)$ as an algebra but as a finitely generated $S$-module, where $S \subset \hat{S}$ is a subring chosen appropriately. Geometrically, we study the image of $\canmod$ under the projection into the (smaller) projective space $\Proj(S)$.

So let $S =  \Bbbk[x_0,x_1,y_0,y_1,y_2,y_3]$ be the graded polynomial ring, where the $x_i$ and $y_j$ are as before with $\deg(x_i)= 2$ and $\deg(y_j)=3$.
The natural homomorphism 
$$f\colon S \rightarrow R(X)$$
gives $R(X)$ the structure of a graded $S$-algebra. In the following, we consider $R(X)$ as a graded $S$-module via the homomorphism $f$. 
\begin{lemma}\label{lem_canringfinggenmod}
	$R(X)$ is a finitely generated $S$-module.
\end{lemma}
\begin{proof}
	By Proposition \ref{prop_tricanbicanbase} and Remark \ref{rem_emptybaseloc}
	the elements $x_i$ and $y_j$ have an empty vanishing locus in $\canmod$, hence $R(X)$ is finite over $S$.
\end{proof}
Using the closed embedding in \eqref{equ_closedemb}, we will from now on identify $\canmod$ with its image in $\P(2^2,3^4,4^4,5^3)$. Now, since $R(X)$ is finitely generated as an $S$-module, 
the homomorphism $f\colon S \rightarrow R(X)$ induces a finite morphism of projective schemes
$$\varphi\colon \canmod  \rightarrow \P(2^2,3^4)$$
with image $Y = \Proj(S_Y) \subset \P(2^2,3^4)$, where $S_Y = S/\ann_S(R(X))$.  
\begin{lemma}\label{lem_normalization}
	$(\canmod, \varphi)$ is the normalization of $Y$. 
\end{lemma}
\begin{proof}
	We already know that $\varphi\colon \canmod \rightarrow Y$ is a finite morphism. Furthermore, $\varphi$ is birational because the 
	tricanonical map $\phi_3 \colon \canmod \dashrightarrow  \P^3$ is birational onto its image and factors over $Y$.  
	Hence, as $\canmod$ is normal, $(\canmod,\varphi)$ is the normalization of $Y$.
\end{proof}
As a next step we will describe the minimal free resolution of $R(X)$ as an $S$-module. First we note that $R(X)$ is a Cohen-Macaulay graded $S$-module, hence  by the Auslander-Buchsbaum formula $R(X)$ has projective dimension 3.  The fact that $R(X)$ is a Gorenstein ring implies the following symmetry condition: 
\begin{proposition}\label{prop_algebracanonical}
	Let
	$$ 0 \leftarrow R(X) \leftarrow F_0 
	\leftarrow \ldots
	\leftarrow F_3 \leftarrow 0$$ be a minimal free resolution of $R(X)$ as an $S$-module, where $F_i = \bigoplus_{j \geq 0} S(-j)^{\beta_{i,j}}$. \\
	Then 
	$$
	\beta_{i,j}  = \beta_{3-i,17-j} \ \text{ for } \ 
	0\leq i \leq 3, \ 0 \leq j \leq 17,$$
	and the Betti numbers are of the following form: 
	\[
	{\rm \begin{matrix}
		&0&1&2&3\\
		\rm{total}:&8&26&26&8\\
		0:&1&.&
		.&.\\1:&.&.&.&\text
		.\\2:&.&.&.&.\\3
		:&.&.&.&.\\4:&4&
		.&.&.\\5:&3&6&.&.\\
		6:&.&12&.&.\\7:&.&8&8&.\\8:&.&.&12&.\\9:&.&.&6&3\\10:&.&.&.&4
		\\
		11:&.&.&.&.\\12
		:&.&.&.&.\\13:&.&.&.&.\\14:&.&.&
		.&1\\\end{matrix}}
	\]
\end{proposition}
\begin{proof}
	Let $\omega_S  \cong S(-16)$  denote the canonical module of $S$. 
	Applying $\Hom_S(\textup{-},\omega_S)$ to a minimal free resolution of $R(X)$ yields a minimal free resolution of $\Ext^3_S(R(X),  \omega_S) \cong \omega_{R(X)}$
	\begin{align*} \scalebox{0.93}{ $0 \leftarrow \omega_{R(X)}
		\leftarrow \Hom_S(F_3,\omega_S) 
		\leftarrow \Hom_S(F_2,\omega_S) 
		\leftarrow \Hom_S(F_1,\omega_S)
		\leftarrow \Hom_S(F_0,\omega_S)$}
	\leftarrow 0 .
	\end{align*}
	On the other hand, as $\omega_{R(X)} \cong R(X)(1)$, we obtain 
	\begin{align*} 0 \leftarrow R(X) 
	&\leftarrow \Hom_S(F_3,S(-17)) 
	\leftarrow \Hom_S(F_2,S(-17)) \\
	& \leftarrow \Hom_S(F_1,S(-17))
	\leftarrow \Hom_S(F_0,S(-17))
	\leftarrow 0,  
	\end{align*} 
	which is another minimal free resolution of $R(X)$ and shows the first claim. To determine the exact Betti numbers we consider the $R(X)$-regular sequence $x_0,x_1,y$ as in the proof of Proposition \ref{prop_canonicalring5}. We know that $M_a = R(X)/(x_0,x_1,y)$ as an $S/(x_0,x_1,y)$-module has the same Betti numbers as $R(X)$  as an $S$-module. Modulo $x_0,x_1,y$ the Artinian module $M_a$ decomposes into a direct sum of three  modules $M_a^{(i)}$ with Hilbert series 
	\begin{align*}
	h_0 =& 1+3t^3, \\
	h_1 =& 4t^4, \\
	h_2 =& 3t^5+t^8
	\end{align*}
	and Hilbert numerators
	\begin{align*}
	h_0\cdot(1-t^3)^3 =& 1-6t^6+8t^9-3t^{12}, \\
	h_1\cdot(1-t^3)^3 =& 4t^4 -12t^7+12t^{10}-4t^{13}, \\
	h_2\cdot(1-t^3)^3 =& 3t^5-8t^8+6t^{11}-t^{17}.
	\end{align*}
	Thus, $M_a$ and $R(X)$ have the Betti numbers as claimed. 
\end{proof}
Thus $R(X)$ has a self-dual resolution of length 3, but we cannot apply the famous structure theorem of Buchsbaum-Eisenbud (\cite{BuchsbaumEisenbud77}), since $R(X)$ is not a cyclic $S$-module. 
However, we can apply the structure result for Gorenstein algebras of \cite{Stenger19}:  
\begin{theorem}[\cite{Stenger19}, Theorem 1.5 and Corollary 5.6]\label{prop_structure}
	There exists a minimal free resolution of $R(X)$ as an $S$-module of type
	\[ 0 \leftarrow R(X) \leftarrow F_0 \xleftarrow{d_1} F_1 \xleftarrow{d_2} F_1^\ast(-17) \xleftarrow{d_1^t} F_0^\ast(-17) \leftarrow 0,\]
	where $d_2$ is alternating.
\end{theorem}
Using this statement we can translate the question of constructing or classifying numerical Godeaux surfaces from a problem in algebraic geometry into a problem in homological algebra. Thus, the main focus of our approach is to construct and describe $S$-modules $R$ having a minimal free resolution as described above.  

We end this section by determining a minimal set of generating relations of $\canmod \subset \P(2^2,3^4,4^4,5^3)$.  Using the proof of the structure theorem in \cite{Stenger19}, these relations can be computed explicitly from a minimal free resolution $\textbf{F}$ of $R(X)$ since the ring structure of $R(X)$ induces a multiplicative structure on the complex $\textbf{F} \otimes\textbf{F}$ and a morphism of  complexes $\textbf{F} \otimes \textbf{F} \rightarrow \textbf{F}$. Using this morphism, we compute a minimal generating set of $R(X)$ with the procedure \href{https://www.math.uni-sb.de/ag/schreyer/images/data/computeralgebra/M2/doc/Macaulay2/NumericalGodeaux/html/_canonical__Ring.html}{canonicalRing} from our package \href{https://www.math.uni-sb.de/ag/schreyer/images/data/computeralgebra/M2/doc/Macaulay2/NumericalGodeaux/html/index.html}{NumericalGodeaux}.

From \eqref{equ_closedemb} we know that there exists a surjective ring homomorphism
$$\hat{f} \colon \hat{S} \rightarrow R(X),$$
where $\hat{S} = \Bbbk[x_0,x_1,y_0,\ldots,y_3,z_0,\ldots,z_3,w_0,w_1,w_2]$ is the graded polynomial ring as defined before. 
Let $r_0 = 1,r_1= z_0,\ldots,r_4= z_3,r_5 = w_0,r_6=w_1,r_7 = w_2$ which generate $R(X)$ as an $S$-module. Proposition \ref{prop_algebracanonical}  shows that there are 26 $S$-linear relations between these module generators: 
\begin{equation}\label{eq_rel1}
0 = \sum_{k=0}^{7} g_{m,k}r_k.
\end{equation}
Furthermore, for the 28 elements $r_ir_j \in R(X)$, $1 \leq i \leq j \leq 7$, there exist elements $s_{i,j,k} \in S$ such that 
\begin{equation}\label{eq_rel2} 
r_ir_j = \sum_{k=0}^{7}s_{i,j,k}r_k.
\end{equation}
These relations are linearly independent and are uniquely determined modulo the relations in \eqref{eq_rel1}. 
Let $I_X \subset \hat{S}$ be the ideal generated by the relations in \eqref{eq_rel1} and \eqref{eq_rel2}. Then:
\begin{lemma}\label{prop_canonicalringideal}
	$R(X) \cong \hat{S}/I_X$
\end{lemma}
\begin{proof}
	Since all generators of $I_X$ define relations in $R(X)$,  $\hat{f}$ factors through a surjective homomorphism $\hat{S}/I_X \rightarrow R(X)$. On the other hand, as an $S$-module, $\hat{S}/I_X$ is also generated by $r_0,\ldots,r_7$, and every relation in \eqref{eq_rel1} is also an $S$-linear relation between the module generators of $\hat{S}/I_X$. Hence, there exists a surjective $S$-linear homomorphism $R(X) \rightarrow \hat{S}/I_X$ which shows the claim.
\end{proof}
We conclude that the homogeneous ideal of  $\canmod \subset \P(2^2,3^4,4^4,5^3)$ is minimally generated by 54 equations in the degrees $(6^6,7^{12},8^{18},9^{12},10^6)$.

\begin{remark}\label{rem_rringcond}
	In Theorem \ref{prop_structure} we have seen that the canonical ring of any numerical Godeaux surface, considered as an $S$-module, admits a minimal free resolution with an alternating middle map. Conversely, given an exact sequence 
	\[F_0 \xleftarrow{d_1} F_1 \xleftarrow{d_2} F_1^\ast(-17) \xleftarrow{d_1^t} F_0^\ast(-17) \leftarrow 0,\]
	where $F_0$ and $F_1$ are defined as above and $d_2$ is alternating, the question is whether the $S$-module $R$ admits a ring structure and, if so, whether $R$ defines the canonical ring of a numerical Godeaux surface. 
	In \cite{Stenger18}, Theorem 5.0.2, we give a sufficient condition for $R$ carrying a ring structure depending only on the first syzygy matrix $d_1$: let $d_1'$ be the matrix obtained from $d_1$ by erasing the first row, and let $I'$ denote the ideal generated by the maximal minors of $d_1'$. Then $R$ admits the structure of a Gorenstein ring if
	\par \smallskip
	\emph{(R.C.) \qquad $\depth(I',S) \geq 5$.}
	\par \smallskip
	If this condition is satisfied, then $ \Proj R$ defines a surface embedded in a weighted projective space. Moreover, under the condition that $\Proj R$ has only Du Val singularities, $\Proj R$ defines indeed the canonical model of a numerical Godeaux surface.
\end{remark}

\section{Normal form and a complete intersection of four quadrics in $\PP^{11}$}\label{normal form}
In this section we will introduce our unfolding techniques for the minimal free resolution of $R(X)$ as an $S$-module. 
We set up the first two syzygy matrices with unfolding parameters and study the system of relations which arises by setting the product of these matrices to zero. Using these relations, we introduce a normal form for the matrices in a minimal free resolution of $R(X)$ as an $S$-module. 

From the last section we know that there exists a minimal free resolution of $R(X)$ of the form 
\[
0 \leftarrow R(X) \leftarrow F_0 \xleftarrow{d_1} F_1 \xleftarrow{d_2} F_1^{\vee} \xleftarrow{d_1^t} F_0^{\vee} \leftarrow 0
\]
with $F_0 = S \oplus S(-4)^4 \oplus S(-5)^3$, $F_1 = S(-6)^6 \oplus S(-7)^{12} \oplus S(-8)^8$, $F_i^{\vee}  = \Hom(F_i,S(-17))$ and $d_2^t= -d_2$.
\begin{notation}[The general set-up]\label{not_gensetup}
	We write 
	\begingroup
	\renewcommand*{\arraystretch}{1.1}
	\begin{align*}
	d_1& =   \begin{array}{c|c|c|c}
	& 6S(-6) & 12S(-7) & 8S(-8) \\ \hline
	S &  \color{red} b_0(y)  \color{blue} + \ast&  \color{blue} \ast &  \color{blue} \ast \\ \hline
	4S(-4) & \color{blue} a  & \color{red} b_1(y) & \color{blue} c \\ \hline
	3S(-5) & \color{black} 0 & \color{blue} e & \color{red}b_2(y)
	\end{array} \\ \\
	d_2& = \begin{array}{c|c|c|c}
	& 6S(-11) & 12S(-10) & 8S(-9) \\ \hline
	6S(-6) & \color{blue} o & \color{blue} n & \color{red}b_3(y) \\ \hline
	12S(-7) & \color{blue} -n^t  & \color{red}{\color{red} b_4(y)} & \color{blue} p \\ \hline
	8S(-8) & \color{red}-b_3(y)^t & \color{blue} -p^t &
	\color{black} 0
	\end{array} 
	\end{align*}
	\endgroup
	The matrices $o$ and $b_4$ are both skew-symmetric. Since there are no elements of degree $1$ in $S$, the maps $S(-5)^3  \leftarrow S(-6)^{6}$ and $S(-8)^8 \leftarrow S(-9)^8$ are both zero.
	The red matrices are the submatrices of $d_1$ and $d_2$ which depend only on the variables $y_0,\ldots,y_3$. More precisely, for each $i$, all entries of $b_i(y)$ are linear combinations of $y_0,\ldots,y_3$ with coefficients in $\Bbbk$. 
	By $d_1'$ we denote the matrix obtained from $d_1$ by erasing the first row. We do not assign names to the matrices indicated by $\ast$ since they won't play a role in the following. For the matrices marked in blue we obtain the characterization:
	\begin{center}
		\renewcommand{\arraystretch}{1.4}
		\begin{tabular}{l|l|l} 
			& degree  & entries \\ \hline
			$\color{blue} a$, $ \color{blue} e$, $\color{blue} p$ & 2 & linear  combinations of $x_0,x_1$ \\ \hline
			$\color{blue} c$, $\color{blue} n$ & 4 & linear combinations of $x_0^2,x_0x_1,x_1^2$ \\ \hline
			$\color{blue} o$ & 5 & linear combinations of $x_iy_j,\ i = 0,1,\ j = 0,\ldots,3$ \\
		\end{tabular}
	\end{center}
\end{notation}
We start by describing the minimal free resolution modulo the regular sequence $x_0,x_1$. 
Let $\bar{R} := R(X)/(x_0,x_1)R(X)$ and $T := S/(x_0,x_1) \cong \Bbbk[y_0,\ldots,y_3]$ with $\deg(y_i) = 3$. 
As a $T$-module $\bar{R}$ splits into a direct sum 
\[\bar{R} =  \bigoplus_{k=0}^{2} \bar{R}^{(k)} \textup{ with } 
\bar{R}^{(k)}  := \bigoplus_{\substack{ j \equiv k \\ ( \mathrm{mod }\ 3)}} \bar{R}_j.\] 
Moreover, for any minimal free resolution $\textbf{F}$ of $R(X)$ the complex $\bar{\textbf{F}} = \textbf{F}/(x_0,x_1)$ decomposes into a direct sum of three $T$-complexes which are minimal free resolutions of $\bar{R}^{(k)}$ as $T$-modules.
\begin{lemma}
	$\Proj(\bar{R}^{(0)})$ is a finite scheme of length 4 in $\P^3$.
\end{lemma}
\begin{proof} We know that 
	$ \bar{R}^{(0)} = \bigoplus_{k \geq 0} \bar{R}_{3k}$ is a graded ring. 
	Furthermore, the minimal free resolution of $\bar{R}^{(0)}$ as a $T$-module is of the form
	\[ 0 \leftarrow \bar{R}^{(0)} \leftarrow T \leftarrow T(-2)^{6} \leftarrow T(-3)^{8} \leftarrow T(-4)^3 \leftarrow 0,\] 
	where we consider the variables $y_j$ with degree 1 now. The Hilbert polynomial of $\bar{R}^{(0)}$ is the 
	constant polynomial $4$, thus $\Proj(\bar{R}^{(0)}) \subset \P^3$ is a finite scheme of length 4. 
\end{proof}

\begin{remark} 
	The Hilbert scheme of length 4 subschemes of $\PP^{3}$ which span $\PP^{3}$ is known to be irreducible with finitely many orbits under the $\PGL(4)$-action. In this paper we focus on the dense orbit which consists of  collections of $4$ distinct points in general position.
	The fact that there are only finitely many orbits give us hope that a complete classification of numerical Godeaux surfaces along the lines of our approach might be possible with further work.
\end{remark}

\begin{proposition}\label{prop_distinct_points} Let $X$ be a numerical Godeaux surface. The following are equivalent
	\begin{enumerate}
		\item[(a)] $|2K_X|$ has no fixed part and four distinct base points.
		\item[(b)] $\Proj(\bar{R}^{(0)})$ consists of four distinct points.
	\end{enumerate}
\end{proposition}

\begin{proof} The scheme $\Proj(\bar{R}^{(0)})$  is the image of the base points of $|2K_{X}|$ under the tricanonical map from $\canmod$ to $\PP^3$. The fixed part $F$ in case $(ii)$ and $(iii)$ of Proposition \ref{prop_bicanoncond} gets contracted to rational double points leading to a non-reduced scheme structure of $\Proj(\bar{R}^{(0)})$. So these cases are excluded under the assumption $(b)$. To prove $(a) \Rightarrow (b)$, we note that a general member of $|M|=|2K_{X}|$ is a smooth non-hyperelliptic curve which passes through the 4 base points and is canonically embedded by $|3K_{X}|$. Thus the points stay distinct. 
\end{proof}
Note that the image points span $\PP^{3}$ since the homogeneous ideal $J$ of  $\Proj(\bar{R}^{(0)})$ in $T$ contains no linear forms.

\begin{definition}
	Let  $X$ be a numerical Godeaux surface satisfying the equivalent conditions of Proposition \ref{prop_distinct_points}.
	A \textit{marking} on $X$ is an enumeration $p_{0},\ldots,p_{3}$ of the base points of $|2K_{X}|$. 
	A \textit{marked} numerical Godeaux surface is a numerical Godeaux surface together with a marking.
\end{definition}

A marked numerical Godeaux surface can only have the torsion groups $0, \ZZ/3\ZZ$ or $\ZZ/5\ZZ$ because a divisor $D\in |K+\tau|$ with $2\tau=0$  leads to a fiber $2D \in |2K_{X}|$ and hence to an everywhere non-reduced base locus.
The moduli space of marked numerical Godeaux surfaces is an $S_{4}$-cover of an open part of the moduli space of numerical  Godeaux surfaces. Introducing a marking allows us to change coordinates such that $|3K_{X}|$ maps the base points to the coordinate points of of $\PP^{3}$:
\[p_0 \mapsto (1:0:0:0),\ p_1 \mapsto (0:1:0:0),\ p_2 \mapsto (0:0:1:0) \textup{ and } p_3 \mapsto (0:0:0:1).\]
The stabilizer $G \leq \Aut(\P^3) = \PGL(4,\Bbbk)$ of the coordinate points as a set is $G \cong (\Bbbk^{\ast})^{3} \rtimes S_4$. 

\begin{proposition} \label{descriptionFmod}
	Let $X$ be a marked Godeaux surface. Then after a change of coordinates of $\PP(2^{2},3^{4},4^{4},3^{5})$ and a change of basis of the minimal free resolution $\bf{F}$ of $R(X)$ we may assume that the summands of
	$$\bar{\bf{F}}= \bf{F}/(x_{0},x_{1})=\bar{\bf{F}}^{(0)}\oplus\bar{\bf{F}}^{(1)}\oplus \bar{\bf{F}}^{(2)}$$ as 
	$T=\Bbbk[y_{0},\ldots,y_{3}]=S/(x_{0},x_{1})$ are the following complexes
	\begin{enumerate}
		\item[(0)] $\bar{\bf{F}}^{(0)}$ is  the \Mac2 resolution 
		$$T \leftarrow 6T(-2\cdot 3) \leftarrow 8T(-3\cdot 3)\leftarrow 3T(-4\cdot 3) \leftarrow 0, $$
		of $T/J$, where 
		\[J=(y_iy_j \mid 0 \leq i < j \leq 3 ) = \bigcap_{i=0}^{3} J_i
		\quad \hbox{ and } \quad J_{i}=(y_{j}| j\not=i).\]
		\item[(1)] $\bar{\bf{F}}^{(1)}$ is up to a twist the sum of 4 Koszul complexes
		$$T(-4) \leftarrow 3T(-4-3) \leftarrow 3T(-4-2\cdot3 )\leftarrow T(-4-3\cdot3) \leftarrow 0 $$
		resolving $T/J_{i}$,
		\item[(2)] $\bar{\bf{F}}^{(2)}= (\bar{\bf{F}}^{(0)})^{\vee}=\Hom(\bar{\bf{F}}^{(0)},T(-17))[-3]$ is up to twist and shift the dual of $\bar{\bf{F}}^{(0)}$. 
	\end{enumerate}
\end{proposition}

\begin{proof} Since $x_{0},x_{1}$ are a regular sequence on $R(X)$ and  on $S$, the complex
	\[
	0 \leftarrow \bar R(X) \leftarrow \bar F_0 \xleftarrow{\bar d_1} \bar F_1 \xleftarrow{\bar d_2} \bar F_1^{\vee} \xleftarrow{\bar d_1^t} \bar F_0^{\vee} \leftarrow 0
	\]
	is still exact. We choose coordinates on $\PP(3^{4})$ such that $$\bar R^{(0)} = T/J \hbox{ with } J=(y_iy_j \mid 0 \leq i < j \leq 3 )$$
	holds. Then we can choose a basis of ${\bf{F}}$ such that $\bar{\bf{F}}^{(0)}$ has the desired shape and $\bar{\bf{F}}^{(2)}= (\bar{\bf{F}}^{(0})^{\vee}$ holds. It remains to normalize $\bar{\bf{F}}^{(1)}$. We start to prove that the sheaf associated to $\bar R^{(1)}$ on $\PP(3^{4})= \PP^{3}$ has support in the coordinate points.
	
	\begin{lemma}\label{lem_annihilator}
		$\ann_T(\bar{R}^{(1)})  = J$.
	\end{lemma}
	\begin{proof}
		We know that
		\[ \ann_T (\bar R)  =\bigcap_{i=0}^{2} \ann_T(\bar R^{(i)})=\ann_T(\bar R^{(0)}) = J\]
		since $\bar R$ is a $T$-algebra, in particular $J \subset \ann_T(\bar R^{(1)})$.  Thus $V(\ann_T(\bar R^{(1)})) \subset V(J)$ is clear, and it is enough to prove equality of the vanishing loci, since $J$ is a radical ideal.
		
		Suppose that there is a coordinate point $p_i$ which is not contained in $V(\ann_T(\bar R^{(1)}))$. Then there exists an integer $n_i \geq 1$ such that $y_i^{n_i} \in  \ann_T(\bar R^{(1)})$. 
		Hence for each $z \in R_{4}=H^{0}(X,4K_{X})$ there is a relation of the form
		\[ r_{0}x_0 + r_{1}x_1 + y_i^{n_i}z  = 0\]
		in $R(X)$, where $r_{0},r_{1} \in R(X)$.
		Since $y_i^{n_i}(p_i) \neq 0$, all forms $z\in H^{0}(X,4K_{X})$ must vanish at the point $p_i \in \canmod$. But 
		$|4K_{X}|$ is base point free by \cite{Bombieri}, Theorem 5.2, a contradiction.
	\end{proof}
	
	Now let $N= \bar R^{(1)}(4)$ and consider the associated sheaf $\widetilde N$ on $\PP(3^{4})$.
	As a module over $T=\Bbbk[y_{0},\ldots,y_{3}]$ with the standard grading $\deg y_{i}=1$ the module  $N$ has a linear resolution
	$$ 0 \leftarrow N \leftarrow T^4 \leftarrow 
	T(-1)^{12} \leftarrow T(-2)^{12} \leftarrow T(-3)^4 \leftarrow 0. $$
	Thus in this grading $N$ is a $0$-regular Cohen-Macaulay module. Hence
	$$N= \oplus_{d\ge 0} H^{0}(\PP^{3},\widetilde N(d)).$$ 
	Since $J$ and $J_{i}$ generate the same ideal in $T_{y_{i}}$, the restriction of $\widetilde N$ to the affine chart $U_{i}=\{y_{i}\not=0\}$
	is $\sO_{p_{i}}^{r_{i}}$ for some $r_{i}\ge 1$ and
	$$\widetilde N \cong \bigoplus_{i=0}^{3} \sO_{p_{i}}^{r_{i}}.$$
	Since $H^{0}(\PP^{3},\widetilde N(d))=N_{d}$ is 4-dimensional for each $d\ge1$ we conclude $r_{i}=1$ for all $i$ and 
	$$N= \bigoplus_{i=0}^{3} T/J_{i}.$$
	Thus $\bar R^{(1)}= \bigoplus_{i=0}^{3} T/J_{i}(-4)$ holds
	with respect to the grading of $S$. Hence for suitable chosen generators for $H^{0}(X,4K_{X})$ and suitable basis of $\bf{F}$ the complex $\bar{\bf{F}}^{(1)}$ has the desired shape.
\end{proof}

\begin{remark} 
	The \Mac2-choice of the resolution of $J_{i}=(y_{j}\mid j\not=i)$ has not a skew-symmetric middle matrix.
	We take the following resolution of $T/J_{0}(-4)$  
	\[T(-4) \xleftarrow{\begin{pmatrix}
		y_1 & y_2 & y_3 
		\end{pmatrix}} T(-7)^3 \xleftarrow{\begin{pmatrix}
		0&{y}_{3}&
		{-{y}_{2}} \\
		-y_3 & 0 & y_1 \\
		y_2 & -y_1 & 0
		\end{pmatrix}} T(-10)^3 \xleftarrow{\begin{pmatrix}
		y_1 \\ y_2 \\ y_3 
		\end{pmatrix}}
	T(-13)^3 \leftarrow 0.\]
	and similarly for the other $T/J_{i}(-4)$. 
	
	The \Mac2-choice of the resolution of $T/J$ has the following differentials
	\begin{align}
	T & \xleftarrow{
		\begin{pmatrix}
		y_0y_1&
		{y}_{0} {y}_{2}&
		{y}_{1}  {y}_{2}&
		{y}_{0}  {y}_{3}&
		{y}_{1}  {y}_{3}&
		{y}_{2}  {y}_{3}
		\end{pmatrix}}  T(-6)^{6} \label{eq_d1mod1}
	\\[13pt]
	T(-6)^{6} & \xleftarrow{\begin{pmatrix}
		{-{y}_{2}}&
		0&
		{-{y}_{3}}&
		0&
		0&
		0&
		0&
		0\\
		{y}_{1}&
		{-{y}_{1}}&
		0&
		0&
		{-{y}_{3}}&
		0&
		0&
		0\\
		0&
		{y}_{0}&
		0&
		0&
		0&
		0&
		{-{y}_{3}}&
		0\\
		0&
		0&
		{y}_{1}&
		{-{y}_{1}}&
		{y}_{2}&
		{-{y}_{2}}&
		0&
		0\\
		0&
		0&
		0&
		{y}_{0}&
		0&
		0&
		{y}_{2}&
		{-{y}_{2}}\\
		0&
		0&
		0&
		0&
		0&
		{y}_{0}&
		0&
		{y}_{1}
		\end{pmatrix}} T(-9)^{8} \label{eq_d1mod2}
	\\[13pt] 
	T(-9)^{8} & \xleftarrow{ 
		{\begin{pmatrix}
			{-{y}_{3}}&0&{y}_{2}&0&{-{y}_{1}}&0&0&0\\
			{-{y}_{3}}&{-{y}_{3}}&{y}_{2}&{y}_{2}&0&0&{-{y}_{0}}&0\\
			0&0&0&{-{y}_{2}}&0&{y}_{1}&0&{-{y}_{0}}\end{pmatrix}}^t
	} T(-12)^{3} \label{eq_d1mod3}
	\end{align}
\end{remark}

\begin{definition}We call any minimal free resolution of the type
	\begin{equation}\label{eq_standardres}
	0 \leftarrow R(X) \leftarrow F_0 \xleftarrow{d_1} F_1 \xleftarrow{d_2} F_1^\vee \xleftarrow{d_1^t} F_0^{\vee} \leftarrow 0
	\end{equation}
	with a skew-symmetric middle matrix $d_2$ and 
	which restricts with $x_0=x_1=0$ to the complex described above a \textit{standard resolution} of $R(X)$. 
\end{definition}
\begin{remark}
	Over an algebraically closed field, the canonical ring of  any marked numerical Godeaux surface admits a standard resolution. Note that such a standard resolution is in general not unique.
\end{remark}

In a standard resolution the red parts of \ref{not_gensetup} are completely determined. We now plan to determine the remaining blue parts of the resolution by using (unfolding) parameters for their entries and analyzing their relations
which are imposed by the condition $d_{1}d_{2}=0$. We start analyzing the relations imposed by
$d_{1}'d_{2}=0$, where $d_{1}'$
is the matrix obtained from $d_1$ by erasing the first row.  The first row of $d_{1}$ will be added at the last step in our construction, since we can recover it from $d_{2}$ by a syzygy computation.

\begin{center}
	\bgroup
	\renewcommand*{\arraystretch}{1.7}
	$0 =  d_1'd_2=\left(\begin{array}{c|c|c} 
	{\color{blue} ao} -{\color{red} b_1(y)}{\color{blue}n^{t}} - {\color{blue} c}{\color{red} b_3(y)^{t}} & {\color{blue} an}\color{black} -{\color{blue} cp^{t}} & {\color{blue} a}{\color{red} b_3(y)}+{\color{red} b_1(y)}{\color{blue} p} \\ \hline
	-{\color{blue} en^{t}} & {\color{blue} e}{\color{red} b_4(y)} -{\color{red} b_2(y)}{\color{blue}p^{t}} & {\color{blue} e}{\color{blue} p} \end{array} \right). $ 
	\egroup
\end{center}

Using the order of $p_0,\ldots,p_3$, we introduce a natural labeling $a_{i,j}^{(k)}$ for the $24$ entries of the matrix $a$ 
as indicated below in the $(1+4)\times (6+12)$ submatrix of $d_{1}$:
\begin{center}
	\begingroup
	\renewcommand*{\arraystretch}{1.6}
	$ \left( \begin{array}{cccccc|cccc}
	\color{red} {y}_{0} \color{red} {y}_{1}\color{blue}+\ast&
	\color{red}{y}_{0} \color{red}{y}_{2}\color{blue}+\ast&
	\color{red} {y}_{1} \color{red} {y}_{2}\color{blue}+\ast&
	\color{red} {y}_{0} \color{red} {y}_{3}\color{blue}+\ast&
	\color{red} {y}_{1} \color{red} {y}_{3}\color{blue}+\ast&
	\color{red} {y}_{2} \color{red} {y}_{3}\color{blue}+\ast& \multicolumn{4}{c}{\color{blue}\ast} \\ \hline
	\color{blue}		a^{(0)}_{0,1}&
	\color{blue}		a^{(0)}_{0,2}&
	\color{blue}		a^{(0)}_{1,2}&
	\color{blue}		a^{(0)}_{0,3}&
	\color{blue}		a^{(0)}_{1,3}&
	\color{blue}		a^{(0)}_{2,3} & \color{red} D_0 & \color{red} 0 & \color{red}0 & \color{red}0 \\ 
	\color{blue}		a^{(1)}_{0,1}&
	\color{blue}		a^{(1)}_{0,2}&
	\color{blue}		a^{(1)}_{1,2}&
	\color{blue}		a^{(1)}_{0,3}&
	\color{blue}		a^{(1)}_{1,3}&
	\color{blue}		a^{(1)}_{2,3} & \color{red}0 & \color{red}D_1 & \color{red}0 & \color{red}0\\
	\color{blue}		a^{(2)}_{0,1}&
	\color{blue}		a^{(2)}_{0,2}&
	\color{blue}		a^{(2)}_{1,2}&
	\color{blue}		a^{(2)}_{0,3}&
	\color{blue}		a^{(2)}_{1,3}&
	\color{blue}		a^{(2)}_{2,3}  & \color{red}0 & \color{red}0 & \color{red}D_2 & \color{red}0\\
	\color{blue}		a^{(3)}_{0,1}&
	\color{blue}		a^{(3)}_{0,2}&
	\color{blue}		a^{(3)}_{1,2}&
	\color{blue}		a^{(3)}_{0,3}&
	\color{blue}		a^{(3)}_{1,3}&
	\color{blue}		a^{(3)}_{2,3}& \color{red}0 & \color{red}0 & \color{red}0 & \color{red}D_3 
	\end{array} \right)$ \endgroup 
\end{center}
with
${D_0} =  \begin{pmatrix}
y_1 & y_2 & y_3
\end{pmatrix}$, 
${D_1} =  \begin{pmatrix}
y_0 & y_2 & y_3
\end{pmatrix}$, 
${D_2} =   \begin{pmatrix}
y_0 & y_1 & y_3
\end{pmatrix}$ and 
${D_3} =  \begin{pmatrix}
y_0 & y_1 & y_2
\end{pmatrix}.$ 
Imposing the condition $d_1'd_2=0$ on the general setup for the matrices, we see that 12 out of the 24 $a$-variables are a priori zero: 
\begin{proposition}\label{prop_matrixazero}
	Let $d_1'$ and $d_2$ be as in Notation \ref{not_gensetup} in our standard form satisfying $d_1'd_2 = 0$. If $k \notin \{i,j\}$, then $a^{(k)}_{i,j} = 0$. Furthermore, every non-zero entry of the matrix $p$ and $e$ can be expressed up to a sign by one of the 12 remaining $a$-variables.
\end{proposition}
\begin{proof} 
	The proof of this statement can be deduced either theoretically or computationally by evaluating the relations ${\color{blue} e}{\color{red} b_4(y)} -{\color{red} b_2(y)}{\color{blue}p^{t}} =0$ and ${\color{blue} a}{\color{red} b_3(y)}+{\color{red} b_1(y)}{\color{blue} p} = 0$.  For a theoretical treatment we refer to \cite{Stenger18}, Section 7.1. A \Mac2 computation  is given by the procedure \href{https://www.math.uni-sb.de/ag/schreyer/images/data/computeralgebra/M2/doc/Macaulay2/NumericalGodeaux/html/_get__Relations__And__Normal__Form.html}{getRelationsAndNormalForm} 
	from our \Mac2 package \href{https://www.math.uni-sb.de/ag/schreyer/images/data/computeralgebra/M2/doc/Macaulay2/NumericalGodeaux/html/}{NumericalGodeaux} (see \cite{SS20}).
\end{proof}
Proposition \ref{prop_matrixazero} shows that for a non-zero $a$-variable, one of the lower indices is equal to the upper index. Thus, to simply our notation, we will denote $a^{(i)}_{i,j} $ by $a_{i,j}$ and $a^{(j)}_{i,j} $ by $a_{j,i}$. 
Consequently, we assume from now on that the $a$-matrix is of the form
\begin{center}
	\begingroup
	\color{blue}
	$a =  \left(\begin{array}{cccccc}
	a_{0,1}&
	a_{0,2}&
	0&
	a_{0,3}&
	0&
	0  \\
	a_{1,0}&
	0&
	a_{1,2}&
	0&
	a_{1,3}&
	0 \\
	0&
	a_{2,0}&
	a_{2,1}&
	0 &
	0&
	a_{2,3}\\
	0&
	0&
	0&
	a_{3,0}&
	a_{3,1}&
	a_{3,2} \end{array}\right)$\endgroup. \end{center}
Furthermore, expressing any $e$- and $p$-variable by one of the $a$-variables, 
we obtain a new matrix of relations:
\begin{center}
	\bgroup
	\renewcommand*{\arraystretch}{1.7}
	$d_1'd_2=\left(\begin{array}{c|c|c} 
	{\color{blue} ao} -{\color{red} b_1(y)}{\color{blue}n^{tr}} - {\color{blue} c}{\color{red} b_3(y)^{tr}} & {\color{blue} an}\color{black} -{\color{blue} cp^{tr}} & 0 \\ \hline
	-{\color{blue} en^{tr}} & 0 & {\color{blue} e}{\color{blue} p} \end{array} \right). $ 
	\egroup
\end{center}
In particular, we see that if the entries of the matrix $a$ are known and satisfy the equation $ep = 0$, all the remaining relations are linear in the unknown
$n, c$ and $o$-variables. 

Now let us recall that a possible entry of the matrix $a$ is a linear combination of $x_0,x_1$ with coefficients in $\Bbbk$. We will think of  these coefficients as Stiefel coordinates, hence as the entries of $2\times 12$-matrices having at least one non-vanishing maximal minor. 
\begin{notation}
	For a matrix $\hat \ell \in \St(2,12)$, we denote by $\ell$ the  line in $\P^{11}$  spanned by the rows of $\hat \ell$.
\end{notation}
An assignment to the 12 remaining $a$-variables gives a matrix $\hat \ell \in \St(2,12)$, and hence a line $\ell \subset \P^{11}$. 
We have to choose lines in $\P^{11}$ such that the quadratic relations coming from $ep = 0$ are satisfied.
Using again our procedure  \href{https://www.math.uni-sb.de/ag/schreyer/images/data/computeralgebra/M2/doc/Macaulay2/NumericalGodeaux/html/_get__Relations__And__Normal__Form.html}{getRelationsAndNormalForm},
we see that there are exactly 4 different forms  
\begin{align*}
q_0 &=  a_{1,2} a_{1,3}\color{black} -a_{2,1} a_{2,3}+a_{3,1} a_{3,2}, \\[8pt]
q_1 &=  a_{0,2} a_{0,3}
-a_{3,0} a_{3,2}
+a_{2,0} a_{2,3}, \\[8pt]
q_2 &= a_{1,0} a_{1,3} -   a_{0,1} a_{0,3}
+a_{3,0} a_{3,1}, \\[8pt]
q_3 &= a_{0,1} a_{0,2}-a_{1,0} a_{1,2}+a_{2,0} a_{2,1}
\end{align*}
which the assignment of the remaining 12 $a$-variables have to satisfy.
The $q_{i}$ are quadrics of rank 6. Hence, there are skew-symmetric matrices $M_0,\ldots,M_3$ of size 4 so that $q_0,\ldots,q_3$ are the Pfaffians of these matrices. One possible choice for these skew-symmetric matrices is: 
\begin{center}
	\begingroup
	\renewcommand*{\arraystretch}{1.4}
	$M_0 = \left( \begin{array}{cccc}0&
	a_{1,2}&
	a_{2,1}&
	a_{3,1}\\
	&0&
	a_{3,2}&
	a_{2,3}\\
	\multicolumn{2}{l}{}  &
	0&
	a_{1,3}\\
	&&&
	0\end{array} \right), \ \ \ \
	M_1 =  \left(  \begin{array}{cccc}0&
	a_{0,2}&
	a_{3,0}&
	a_{2,0}\\
	&
	0&
	a_{2,3}&
	a_{3,2}\\
	\multicolumn{2}{l}{}  &
	0&
	a_{0,3}\\
	&&&
	0 \end{array} \right),$
	\endgroup
\end{center}
\bigskip
\begin{center}
	\begingroup
	\renewcommand*{\arraystretch}{1.4}
	$M_2 = \left( \begin{array}{cccc} 0&
	a_{1,0}&
	a_{0,1}&
	{a_{3,0}}\\
	&
	0&
	a_{3,1}&
	a_{0,3}\\
	\multicolumn{2}{l}{} &
	0&
	a_{1,3}\\
	&&&
	0\end{array} \right), \ \
	M_3 = \left(  \begin{array}{cccc} 0&
	a_{0,1}&
	a_{1,0}&
	a_{2,0}\\
	&
	0&
	a_{2,1}&
	a_{1,2}\\
	\multicolumn{2}{l}{} & 
	0&
	a_{0,2}\\
	&&& 
	0\end{array} \right).$
	\endgroup
\end{center}

The corresponding variety $Q = V(q_0,\ldots,q_3) \subset \P^{11}$  is an irreducible complete intersection. 
Indeed, the 7-dimensional variety $Q$ is irreducible because its singular loci has codimension $8>5$ in $\PP^{11}$. 
The first step of our construction method for a marked numerical Godeaux surface $X$ is to choose a line $\ell \subset Q$.

\begin{lemma}\label{lem_lineQ} Let $\hat \ell \in \St(2,12)$ be a matrix. Then
	\[e((x_{0},x_{1})\hat \ell)p((x_{0},x_{1})\hat \ell) = 0 \Longleftrightarrow \ell \subset Q\subset \P^{11}.\]
\end{lemma}
\begin{proof} Clear from the definition of  the variety $Q \subset \P^{11}$.
\end{proof}

After the choice of a parametrized line $\ell \subset Q$, hence the choice of $a$, $e$ and $p$, 
the second step of our construction consists in solving the remaining equations
\begin{eqnarray*}
	{\color{red} a}{\color{blue}o} -{\color{red} b_1(y)}{\color{blue}n^{t}} - {\color{blue} c}{\color{red} b_3(y)^{t}}&=&0,  \\
	{\color{red} a}{\color{blue}n}\color{black} -{\color{blue} c}{\color{red}p^{t}}&=&0, \\
	-{\color{red}e}{\color{blue} n^{t}} &=&0.
\end{eqnarray*}
This is a system of linear equations for the $c$-, $o$- and $n$-variables.

We end this section by introducing a normal form for the $o$-matrix which is the $6\times6$ skew-symmetric submatrix of $d_2$ whose entries are homogeneous of degree $5$.  Let 
\[
0 \leftarrow R(X) \leftarrow F_0 \xleftarrow{d_1} F_1 \xleftarrow{d_2} F_1^{\vee} \xleftarrow{d_1^t} F_0^{\vee} \leftarrow 0
\]
be a standard resolution of $R(X)$ and define the maps
\[ \alpha_0 = \id_{F_0}, \ \alpha_1 =  \begin{pmatrix} \id_6 & & \gamma \\ & \id_{12} & \\ & & \id_8 \end{pmatrix}, \ \alpha_2 =  \begin{pmatrix} \id_6 & &  \\ & \id_{12} & \\ - \gamma & & \id_8 \end{pmatrix} \textup{ and } \alpha_3 = \id_{F_0^{\vee}},\] where 
$\gamma$ is a $6\times 8$ matrix whose entries are  linear forms in $x_0,x_1$. Then, setting $e_1 = \alpha_0d_1\alpha_1^{-1}$ and $e_2 = \alpha_1d_2\alpha_2^{-1}$, we obtain another isomorphic standard resolution of $R(X)$
\[
0 \leftarrow R(X) \leftarrow F_0 \xleftarrow{e_1} F_1 \xleftarrow{e_2} F_1^{\vee} \xleftarrow{e_1^t} F_0^{\vee} \leftarrow 0,
\]
with a new skew-symmetric matrix   $o+ b_3\gamma^{t}-\gamma b_3^{t}$. Motivated by this, we can reduce the original $o$-matrix by the $6\times8$ matrix $b_3(y)$ and its transpose in a way so that we keep the skew-symmetry and obtain a new matrix which depend only on 12 $o$-variables instead of $60 = 4\cdot\binom{6}{2}$ bilinear $o$-variables 
\[\begin{pmatrix}
0&{o}_{1,0,0}{y}_{0}&{o}_{2,0,1}{y}_{1}&{o}_{3,0,0}{y}_{0}&{o}_{4,0,1}{y}_{1}&0\\
&0&{o}_{2,1,2}{y}_{2}&{o}_{3,1,0}{y}_{0}&0&{o}_{5,1,2}{y}_{2}\\
&&0&0&{o}_{4,2,1}{y}_{1}&{o}_{5,2,2}{y}_{2}\\
&&&0&{o}_{4,3,3}{y}_{3}&{o}_{5,3,3}{y}_{3}\\
&&&&0&{o}_{5,4,3}{y}_{3}\\
&&&&&0\end{pmatrix}.\]
For more details we refer again to our procedure \href{https://www.math.uni-sb.de/ag/schreyer/images/data/computeralgebra/M2/doc/Macaulay2/NumericalGodeaux/html/_get__Relations__And__Normal__Form.html}{getRelationsAndNormalForm}.
Now using  the relations coming from ${\color{red} a} {\color{blue} o} -{\color{red} b_1(y)}{\color{blue}n^{tr}} - {\color{blue} c}{\color{red} b_3(y)^{tr}}=0$, we can express any $n$-variable as a linear combination of $a, o$ and $c$-variables. The final system of linear equations depend then on 12 $o$-variables and 20 $c$-variables, and a solution for these variables can be computed using syzygies. We will study these linear solution spaces more extensively in the following sections.

\begin{remark}
	A different normalization pursued in \cite{Stenger18} is to make the $c$-matrix as much zero as possible. In particular, for torsion-free numerical Godeaux surfaces with no hyperelliptic bicanonical fibers, the $c$-matrix can be assumed to be zero in this setting.
\end{remark}

\section{The Fano variety of lines $F_1(Q)$}\label{Fano}
In this section, we will explain how to construct lines contained in  the complete intersection variety $Q$ and analyze the geometry the Fano variety of lines $F_1(Q)$. Moreover, we explain how the action of  stabilizer group
of the coordinate points  of $\PP^3$ extends to an action on $Q$ respectively $F_1(Q)$. 

Let us start by presenting a method for computing (random) lines in $Q$. Recall that $Q = V(q_0,\ldots,q_3)$ where each $q_i$ is the Pfaffian of a skew-symmetric $4 \times 4$ matrix $M_i$. Let $b_i \in \Bbbk^4$. Then, as $M_i$ is skew-symmetric, $M_ib_i$ gives in general three linear equations in the $a$-variables whose ideal contains the Pfaffian $q_i$. Thus, choosing three out of the four skew-symmetric matrices, say $M_0,M_1$ and $M_2$ and non-zero vectors $b_0,b_1$ and $b_2$ in $\Bbbk^4$, we obtain in general a codimension 9  linear space in $\PP^{11}$, thus a plane $\Lambda \subset \PP^{11}$. Then 
$$ Q \cap \Lambda = V(q_3) \cap \Lambda = \tilde{Q_3} \subset \Lambda \cong \PP^2$$
gives a conic  $\tilde{Q_3}$ in $\PP^2$. Now, over a finite field $\Bbbk = \ZZ/p$, the conic $\tilde{Q_3}$ decomposes into a union of two lines with a probability $1/p$.  Consequently, choosing repetitively vectors $b_0,b_1,b_2$, we end up with a line in $\tilde{Q_3} \subset \Lambda$, and thus, a line $\ell \subset Q$.   


Based on the presented construction of lines in $Q$, we  can see that $F_1(Q)$ is birational to a subscheme of four copies of $\PP^3$.  

Let $\ell \subset Q$ be a line given by a matrix $\hat{\ell} \in \St(2,12)$. 
For a general line in $Q$ the skew-symmetric matrix $m_i = M_i((x_0,x_1)\hat{\ell}))$ has rank 2 and its cokernel has expected free resolution
$$ 0 \leftarrow \coker m_i \leftarrow \Bbbk[x_0,x_1]^4 \leftarrow \Bbbk[x_0,x_1](-1)^4 \leftarrow \Bbbk[x_0,x_1](-1) \oplus \Bbbk[x_0,x_1](-2) \leftarrow 0.$$
Thus, for a general line $\ell \subset Q$, we obtain four constant syzygy vectors $b_i \in \Bbbk^4$. 
\begin{proposition}\label{4P3s}
	With the notation as above, the map 
	\begin{align*}
	F_1(Q) & \dashrightarrow \PP^3 \times \PP^3 \times  \PP^3 \times  \PP^3, \\
	\ell & \mapsto (b_0,b_1,b_2,b_3)
	\end{align*}
	is birational onto its image. Choosing just three out of the four ${\mathbb{P}^3}s$, the corresponding projection gives a generically $2:1$ rational map onto a hypersurface $H$ of type $(4,4,4)$ in $\mathbb{P}^3 \times \mathbb{P}^3 \times \mathbb{P}^3$. 
\end{proposition}

\begin{proof} Most of the statement is clear by the preceding remarks. To see that $F_1(Q)  \dashrightarrow \PP^3 \times \PP^3 \times  \PP^3 \times  \PP^3$ is birational onto its image, it suffices to find a line $\ell$ such that
	\begin{enumerate}
		\item all $M_{i}|_{\ell}$ have a unique constant syzygy $b_{i}$ up to scalars, and
		\item the equations $M_{i}b_{i}=0$ for $i=0,1,2,3$ redefine $\ell$.
	\end{enumerate}
	It is easy to find such example over a finite field. Note that this rational map is not a morphism in neither direction
	because both conditions can fail.
	
	It remains to compute the degree of the hypersurface $H$.
	Let $A=\Bbbk[a_{ij}]$ and $B=\Bbbk[b_{ij}]$ denote the coordinate ring of $\PP^{{11}}$ and the Cox ring of
	$\PP^{3}\times \PP^{3} \times \PP^{3}$. Write $M_{i}b_{i}= N_{i}a$, where $N_i$ is a $4\times 12$ matrix depending only on the $b_i$ and $a$ denotes the column vector of $a$-variables. Note that $b_{i}^{t}N_{i}=0$ because  $b_{i}^{t}M_{i}b_{i}=0$ by the skew-symmetry of $M_{i}$. Then
	$$N=\begin{pmatrix}N_{0}\cr N_{1} \cr N_{2}\end{pmatrix}$$
	is a $12 \times 12$ matrix whose kernel $\ker N(b) \cong \Bbbk^{3}$  for general $b=(b_{0},b_{1},b_{2}) \in \PP^{3}\times \PP^{3} \times \PP^{3}$ corresponds to the plane $\PP^{2}_{b} \subset \PP^{11}$ defined by the nine linear equations obtained from $M_{i}b_{i}=0$ for $i=0,1,2$. Consider the complex
	$$ 0 \leftarrow B(2)^{3} \leftarrow B(1)^{12}\oplus B^{12} \leftarrow B(-1)^{12}\oplus B^{12} \leftarrow B(-2)^{3}\leftarrow 0$$
	with differentials
	$$ \begin{pmatrix}b_{0}^{t}  & 0 & 0&0 \cr 0 &b_{1}^{t} & 0 &0 \cr  0& 0 &b_{2}^{t}  &0\cr \end{pmatrix},
	\begin{pmatrix} 0 & N\cr N^{t} & h \cr \end{pmatrix} \hbox{ and } 
	\begin{pmatrix} b_{0} & 0 & 0 \cr 
	0 &b_{1}& 0 \cr 0 & 0&b_{2}\cr
	0 & 0 & 0  \end{pmatrix}$$
	where $h$ is the Hessian of $q_{3}$. The determinant of this complex, equivalently the annihilator of the $H_{1}$-homology of this complex,
	describes the locus of points $b$ where the restriction of $q_{3}$ to $\PP^{2}_{b}$ has rank $<3$. This is a hypersurface of total degree $12=-3\cdot 2+12\cdot 1- 12\cdot(-1)+3\cdot(-2)$ in the Cox variables. In the fine grading the hypersurface has multidegree 
	$$(4,4,4)=-(2,2,2)+4(1,1,1)-4(-1,-1,-1)+(-2,-2,-2).$$
\end{proof}

\begin{corollary}
	The general fiber 
	\begin{align*}
	F_1(Q) & \dashrightarrow \PP^3 \times \PP^3  \\
	\ell & \mapsto (b_0,b_1)
	\end{align*}
	is an Abelian surface. The general fiber of the projection
	\begin{align*}
	H & \dashrightarrow \PP^3 \times \PP^3  \\
	\ell & \mapsto (b_0,b_1)
	\end{align*}
	is a $16$-nodal quartic in $\PP^{3}$, the Kummer surface of the Abelian surface above.
\end{corollary}

\begin{proof} $\PP^{2}_{b}$'s as  introduced in the previous proof  along which $q_{3}$ has rank $\le 1$ are expected to occur in codimension $3$.
	We can count their number with multiplicities either by enumerative methods using the complex above  or by an experiment over a finite field, since their number when finite, is locally constant in families. The experiment over a finite field allows in addition to establish that the general fiber $H_{b_0,b_1}$ of $H \dashrightarrow \PP^3 \times \PP^3$ has 16 distinct $A_{1}$ singularities. Thus, $H_{b_{0},b_{1}}$
	is a Kummer surface and the double cover is the corresponding abelian surface since it is
	unramified away from the 16 double points. 
	
	The first statement follows also from Reid's result: For general $(b_{0},b_{1}) \in \PP^{3}\times \PP^{3}$
	$M_{0}b_{0}=0,M_{1}b_{1}=0$ defines a $\PP^{5}_{b_{0},b_{1}} \subset \PP^{11}$ and $Q$ restricted to this $\PP^{5}$ is a complete intersection of two quadrics. The variety of lines contained in a smooth complete intersection of two quadrics in $\PP^{5}$
	is isomorphic to the Jacobian of a hyperelliptic curve of genus 2 (see \cite{ReidThesis}).
\end{proof}
\begin{remark}
	To construct Godeaux surfaces over number fields of degree 8 over $\QQ$, we can proceed as follows. 
	\begin{enumerate}
		\item Choose two rational points $b_0, b_1$ in the first two $\PP^3$s and a line in the third $\PP^3$. 
		\item The fiber $H_{b_0,b_1}$ intersects the line in four points. Thus, we obtain a point $(b_0,b_1,b_2)$ in $H$ defined over a degree four number field $K$. 
		\item For a general choice $(b_0,b_1,b_2)$ the last quadric $q_3$ restricted to the corresponding $\PP^2$ is a union of two lines defined over a degree 2 extension field of $K$. 
	\end{enumerate}
	A different construction method of lines in $Q$ defined over a  number field of degree 8 can be found \cite{Stenger18}, Section 7.2. 
\end{remark}

Recall from Section 3 that the subgroup $G \cong  (\Bbbk^{\ast})^{4}/\Bbbk^\ast \rtimes S_4 \subset \PGL(4,\Bbbk)$ fixes the four coordinate points of $\PP^3$ as a set. 
\begin{proposition}
	The induced action of   $(\lambda_0,\ldots,\lambda_3)  \in \Lambda = (\Bbbk^*)^4$ on the underlying vector space of our $\PP^{11}$ and is given by
$$a_{i,j} \mapsto \lambda_i\lambda_j^2a_{i,j}$$ 
and leaves $ Q \subset \PP^{11}$  invariant. 
The induced $S_4$-action on $\PP^{11}$ and $Q$ is a signed permutation action of the indices generated by 
\begin{align*}
&\left(a_{3,2},\,a_{3,1},\,a_{3,0},\,a_{2,3},\,a_{2,1},\,a_{2,0},\,a_{1,3},\,a_{1,2},\,a_{1,0},\,a_{0,3},\,a_{0,2},\,a_{0,1}\right)  \\
(0,1)\colon \qquad \longmapsto &
 \left(-a_{3,2},\,a_{3,0},\,a_{3,1},\,-a_{2,3},\,a_{2,0},\,a_{2,1},\,a_{0,3},\,a_{0,2},\,-a_{0,1},\,a_{1,3},\,a_{1,2},\,-a_{1,0}\right),\\
(1,2)\colon \qquad \longmapsto
& \left(a_{3,1},\,a_{3,2},\,-a_{3,0},\,a_{1,3},\,-a_{1,2},\,a_{1,0},\,a_{2,3},\,-a_{2,1},\,a_{2,0},\,-a_{0,3},\,a_{0,1},\,a_{0,2}\right),\\
(2,3)\colon \qquad \longmapsto
& \left(-a_{2,3},\,a_{2,1},\,a_{2,0},\,-a_{3,2},\,a_{3,1},\,a_{3,0},\,a_{1,2},\,a_{1,3},\,-a_{1,0},\,a_{0,2},\,a_{0,3},\,-a_{0,1}\right),
\end{align*}
corresponding to the transpositions $(0,1)$, $(1,2)$ and $(2,3)$. 
\end{proposition}
\begin{proof}
	By a direct inspection we see that $Q$ is invariant under this action. Note that the action requires signs because of the signs in the Pfaffians. This action is compatible with $S_4$-action on the $y$-variables generated by 
	\begin{align*}
	(0,1)\colon (y_0,y_1,y_2,y_3)&\mapsto (-y_1,-y_0,y_2,y_3), \\
	(1,2)\colon (y_0,y_1,y_2,y_3)&\mapsto (y_0,-y_2,-y_1,y_3), \\
	(2,3)\colon (y_0,y_1,y_2,y_3)&\mapsto (y_0,y_1,-y_3,-y_2).
	\end{align*}
	On the generators $z_0,\ldots,z_3$ the corresponding action is given by $z_i \mapsto \textup{sign}(\sigma)z_{\sigma(i)}$. With respect to these actions, the $a$-matrix is invariant. 
\end{proof}
Restricting to a marked numerical Godeaux surface this reduces to an action of the torus $\Lambda$ and an induced action of $\Lambda$ on the Fano variety of lines $F_1(Q)$. 
\begin{remark} The question whether the geometric quotient $F_1(Q)/\Lambda$ is unirational or not remains open. Embedding $F_1(Q)$ in $\P^{65}$ via the Pluecker embedding and considering a linearization of the given action of $\Lambda$ on $F_1(Q)$, we obtained that the geometric quotient has expected dimension 5 by exhibiting stable points in $F_1(Q)$ (see \cite{Stenger18}, Theorem 8.3.17).
	\medskip
	
	We tried to study $F_1(Q)$ via the rational map from Proposition \ref{4P3s}.
	We were able to compute generators of the multigraded ideal of the image $F'$ of $F_1(Q)$ in $(\PP^3)^4$ and the hypersurface $H$. Moreover, the $\Lambda$-operation on the $a$-variables lifts to an operation on the $b$-variables. We were able to define various rational maps $F' \dashrightarrow \PP^n$ given by multihomogeneous monomials with respect to the $\ZZ^4 \times \ZZ^4$-grading which collapse precisely the $\Lambda$-orbits.  Thus,  $F_1(Q)$ has negative Kodaira dimension.

Using Macaulay2, we computed a birational model of the 5-dimensional quotient space as an anticanonical divisor in a toric variety. However, we were not able to control the singularities of this model good enough to conclude that it is a Calabi-Yau fivefold. For details of this computation we refer to our Macaulay2-package \cite{SS20}. 
\end{remark}

\section{The dominant component}\label{subsec C-complex}
In Section \ref{normal form} we have seen that our construction method of numerical Godeaux surfaces consists mainly of two big steps: first, choosing a line in the quadratic complete intersection $Q \subset \P^{11}$ and second, choosing a solution for a linear system of equations.
This section concerns the second step. We summarize some of the main results of this section.

\begin{theorem}\label{dominant component}
	For a general line $\ell \subset Q \subset \PP^{11}$ the linear system of equations for the  remaining $o$- and $c$-variables has a $4$-dimensional solution space, and $\ell \in F_{1}(Q)$ is a smooth point in the Fano variety of lines in $Q$.
\end{theorem}

\begin{corollary} There exists a $8+3$-dimensional irreducible family in the unfolding parameter space for Godeaux surfaces, which modulo the $(\CC^{*})^{3}$-action gives a $8$-dimensional locally complete family of Godeaux surfaces with trivial fundamental group hence torsion group $\Tors=0$.
\end{corollary}
\begin{remark}
	Note that the Barlow surface is a member of this family by Proposition \ref{prop_barlow} below. Hence, all surfaces in the constructed family are simply connected. 
\end{remark}

We call this family the \emph{dominant family}, because it dominates the Fano variety of lines $F_{1}(Q)$
(\cite{Stenger18} established the irreducibility of $F_{1}(Q)$ numerically).

A sufficient condition for a line $\ell$ to lead only to Godeaux surfaces in the dominant family is that  $\ell$ does not intersect the homology loci of the complexes $C_{1}$ and $C_{2}$ introduced below.
A key point is that the Barlow surfaces are part of the dominant family.
Lines leading to torsion $\ZZ/3\ZZ$- and $\ZZ/5\ZZ$-Godeaux surfaces have to intersect some of the homology loci non-trivially. \medskip

Starting with matrices $d_1'$ and $d_2$ in normal form, the remaining unfolding parameters satisfy exactly 46 homogeneous relations  coming from $d_1'd_2 = 0$ :
the four quadratic relations $q_0,\ldots,q_3$ which are Pfaffians and 42 relations which are linear in the unknown $o$-variables and $c$-variables (see \href{https://www.math.uni-sb.de/ag/schreyer/images/data/computeralgebra/M2/doc/Macaulay2/NumericalGodeaux/html/_get__Relations__And__Normal__Form.html}{getRelationsAndNormalForm}). 

We start with representing these 42 relations by a matrix. By $\underline{c}$ we denote the column vector of the 20 unknown $c$-variables and by $\underline{o}$ the column vector of the remaining 12 $o$-variables.  
The linear system of equations is of the form 
\begin{equation*}
\left(\begin{array}{c| c}
0  & l_1 \\ \hline l_2 & q
\end{array}\right) 
\left(\begin{array}{c}
$\underline{c}$ \\  $\underline{o}$
\end{array}\right) = 0
\end{equation*}
where $l_1$ is a $12 \times 12$-matrix and $l_2$ is a $30 \times 20$-matrix, both having entries linear in the $a$-variables, and $q$ is a 
$30 \times 12$-matrix with quadratic entries. We denote  the $42\times 32$-matrix by $m_a$ and 
the standardly graded polynomial ring with the 12 remaining $a$-variables by $S_a$.

First let us describe the quadratic matrix $l_1$. Arranging the $o$-variables and the corresponding 12 relations of degree 4 properly, $l_1$ is a skew-symmetric matrix which is the direct sum of the  matrices
\begin{equation*}
\begin{pmatrix}
0&a_{3,1}&a_{3,0}\\
&0&a_{3,2}\\
&&0\\
\end{pmatrix}, 
\begin{pmatrix}
0&a_{2,1}&a_{2,0}\\
&0&a_{2,3}\\
&&0\\
\end{pmatrix},
\begin{pmatrix}
0&a_{1,3}&a_{1,0}\\
&0&a_{1,2}\\
&&0\\
\end{pmatrix},
\begin{pmatrix}
0&a_{0,2}&a_{0,1}\\
&0&a_{0,3}\\
&&0\\
\end{pmatrix}.
\end{equation*}
Note that the entries of such a $3\times 3$-matrix are exactly the entries of a row of the $a$-matrix. 
Resolving $l_1$ and $l_1^t = -l_1$ yields a complex $C_{1}$ which is a direct sum of four Koszul complexes 
\[
\begin{matrix}
&0&1&2&3\\\text{total:}&4&12&12&4\\\text{-3:}&4&12&12&4\\
\end{matrix}.
\]

The $30\times 20$-matrix $l_2$ is not as easy to describe. Arranging the $c$-variables and the 30 relations of degree 6, we get 
\begin{equation*}
l_2 = \left(\begin{array}{c| c}
l_1  & 0 \\ \hline n_1 & n_2
\end{array}\right),
\end{equation*}
where $n_1$ is $18\times 12$-matrix and $n_2$ a $18\times 8$-matrix, both having full rank. Furthermore, the matrix $l_2$ has full rank 20 and hence no non-trivial syzygies. However, over the quotient ring $S_Q = S_a/I(Q)$ we obtain syzygies. Moreover, 
rather unexpectedly, both modules $\ker(l_2\otimes S_Q) \cong 2S_{Q}(-1)$ and $\ker(l_2^t\otimes S_Q)\cong 12S_{Q}(-4)$ are free  over this quotient ring. Putting these together, we get a generically exact complex $C_{2}$ with Betti numbers 
\[ \begin{matrix}
&\text{0}&\text{1}&\text{2}&3\\
\text{total:}&12&30&20&2\\
\text{-4:}&12&30&20&\text{.}\\
\text{-3:}&\text{.}&\text{.}&\text{.}&\text{.}\\
\text{-2:}&\text{.}&\text{.}&\text{.}&2\\
\end{matrix}.
\]
Similarly, $\ker(m_{a}^{t}\otimes S_{Q}) \cong 4 S_{Q}(-3)\oplus 12 S_{Q}(-4)$ is free, however $E=\ker(m_{a}\otimes S_{Q})$ is not free. \medskip

\noindent
{\it Proof} of Theorem \ref{dominant component}.
Since any syzygy of $l_2$ induces a syzygy of $m_a$ and every syzygy of $m_a$ projects to one of $l_1$, we can combine these  complexes into a commutative diagram
\begin{center}
	\begin{tikzpicture}\label{C-complex}
	\matrix(m)[matrix of math nodes,
	row sep=2.1em, column sep=2.9em,
	text height=1.5ex, text depth=0.25ex]
	{ & 0 & 0 & 0& \\
		0 &2S_Q(-1) & E & 4S_Q& \\
		0&20S_Q(2) & 20S_Q(2) \oplus 12S_Q(1) & 12S_Q(1) & 0\\
		0&30S_Q(3) & 12S_Q(2) \oplus 30S_Q(3) & 12S_Q(2) &0\\
		0&12S_Q(4) & 4S_Q(3) \oplus 12S_Q(4) & 4S_Q(3)&0\\
		&0 & 0 & 0 &\\};
	\path[->]
	(m-2-1) edge (m-2-2)
	(m-3-1) edge (m-3-2)
	(m-4-1) edge (m-4-2)
	(m-5-1) edge (m-5-2)
	(m-3-4) edge (m-3-5)
	(m-4-4) edge (m-4-5)
	(m-5-4) edge (m-5-5)
	(m-1-2) edge (m-2-2)
	(m-2-2) edge (m-3-2)
	edge (m-2-3)
	(m-3-2) edge node[right] {$l_2$} (m-4-2)
	edge (m-3-3)
	(m-4-2) edge[red] node[right] {$g_2$} (m-5-2)
	edge[red] (m-4-3)
	(m-5-2) edge (m-6-2)
	edge (m-5-3)
	(m-1-3) edge (m-2-3)
	(m-2-3) edge (m-3-3)
	edge (m-2-4)
	(m-3-3) edge[red]  node[right]{$m_a$} (m-4-3)
	(m-4-3) edge (m-5-3)
	edge (m-4-4)
	(m-5-3) edge (m-6-3)
	edge (m-5-4)
	(m-1-4) edge (m-2-4)
	(m-2-4) edge[red](m-3-4)
	(m-3-4) edge node[right] {$l_1$} (m-4-4)
	(m-3-3)  edge[red]  (m-3-4)
	(m-4-4) edge node[right] {$g_1$}  (m-5-4)
	(m-5-4) edge (m-6-4);
	
	\end{tikzpicture} 
\end{center}
with three split exact rows.
The columns are only generically exact. Following the red arrows in the diagram, we obtain a boundary map 
$4S_Q \rightarrow 12 S_Q(4)$.  
We verify computationally that this boundary map is the zero homomorphism, and hence we obtain a complex
\begin{equation}\label{eq_solspace}
0 \rightarrow 2S_Q(-1) \rightarrow E \rightarrow 4S_Q \rightarrow 0
\end{equation}
which is exact except at the last position, due to the homology $H_{1}(C_{2})=\ker(g_{2})/ \im(l_{2})$.
Below we will prove computationally that the support of the homology groups $H_{i}(C_{1})$ and $H_{i}(C_{2})$
have codimension $\ge 2$ in $Q$. Thus a general line $\ell\subset Q$ does not intersect this locus. If $\ell$ is such a line,
then the complex in \eqref{eq_solspace} restricted to $\ell$ is exact and we get an exact sequence of global sections on $\PP^1 \cong \ell$ 
\begin{equation}
0 \rightarrow H^0(\sO_{\PP^1}(-1)^2) \rightarrow H^0(E|_{\ell}) \rightarrow H^0(\sO_{\PP^1}^4) \rightarrow 0. 
\end{equation}
Consequently, we obtain a 4-dimensional solution space for the remaining  $o$- and $c$-variables in this case. 

We will prove that $F_{1}(Q)$ is smooth at a general $\ell$ computationally by studying a Barlow line in Section \ref{subsec_barlow}. Thus with these two further assertions we obtain the proof of Theorem \ref{dominant component}. 
\qed \medskip

In the following we briefly state our results on the homology loci of the complexes introduced above using \Mac2. For a more detailed study we refer to our forthcoming paper studying special bicanonical fibers. 
The complex $C_1$ has just homology at the zeroth position and we compute that
$H_0(C_1) = \coker g_1$ is supported at a 4-dimensional scheme  decomposing in $Q$ into eight irreducible varieties. 

For the second complex $C_2$ we have $H_2(C_2) = H_3(C_2) = 0$ by construction. The module $H_1(C_2) = \ker g_2 / \im l_2$ is supported at a  5-dimensional scheme of degree $72$ in $Q$ consisting of 24 irreducible components all of codimension 6 in $\PP^{11}$.


The module $H_0(C_2) = \coker g_2$ is as well supported at a 5-dimensional scheme of degree $72$ in $Q$ whose minimal primes are exactly the minimal primes of the ideal of  $4\times 4$-minors of the $a$-matrix in $Q$. 
In particular we deduce from these computations that the homology loci of $C_{1}$ and $C_{2}$ have codimension $\ge 2$ in $Q$. 

\subsection{Barlow surfaces}\label{subsec_barlow}
The Barlow surface (\cite{Barlow}) was the first example of a simply connected numerical Godeaux surface. In this subsection we first sketch the original construction due to Barlow, reconstruct the surface then with our construction and show in the end that the Barlow surfaces are deformation equivalent to the members of our dominant family. In particular, every surface in our dominant family is simply connected. 

For the reconstruction of the original Barlow surface we follow the descriptions in \cite{Barlow} and \cite{Lee}.
To start with, we recall that there is an 8-dimensional family of numerical Godeaux surfaces with $\Tors = \ZZ/5\ZZ$ which are given as the quotient of quintics in $\P^3$ under a free action of $\ZZ/5\ZZ$. In \cite{CataneseBabbage}, Catanese showed that there is a 4-dimensional subfamily in which the corresponding quintic is the determinant of a $5\times 5$ symmetric matrix. 
Moreover, in this 4-dimensional family there exists a 2-dimensional subfamily in which the group action of $\ZZ/5\ZZ$ can be extended to a group action of the dihedral group $D_5$. Using a twist of this action, Barlow realized a simply connected numerical Godeaux surface as a quotient of a double cover of such a quintic. This construction shows the existence of a 2-dimensional family of simply connected numerical Godeaux surfaces. 
\par  \smallskip
In the following we briefly recall the description of a symmetric determinantal quintic $Q_5 \subset \P^{3}$ and the definition of the action of $D_5$ on $Q_5$. 
Let $u_1,\ldots,u_4$ denote the coordinates of $\P^3$, and let $\xi$ be a primitive fifth root of unity. Then the group $D_5 = \langle \beta,\tau\rangle$ acts on $\P^3$ via
\begin{equation*}
\begin{aligned}
\beta \colon (u_1:u_2:u_3:u_4) &\mapsto (\xi u_1,\xi^2 u_2,\xi^3 u_3,\xi^4u_4), \\
\tau \colon (u_1:u_2:u_3:u_4) &\mapsto (u_4:u_3:u_2:u_1).
\end{aligned}
\end{equation*}
A quintic in $\P^3$ which is invariant under this action is the determinant of the symmetric matrix
\begin{equation}\label{eq_symmatrix}
A  = (a_{ij})= \begin{pmatrix}
0 & a_1u_1 & a_2u_2 & a_2u_3 & a_1u_4 \\
a_1u_1 & a_3u_2 & a_4u_3 & a_5u_4 & 0 \\
a_2u_2 & a_4u_3 & a_6u_4 & 0 & a_5u_1 \\
a_2u_3 & a_5u_4 & 0 & a_6u_1 & a_4u_2 \\
a_1u_4 & 0 & a_5u_1 & a_4u_2 & a_3u_3  
\end{pmatrix},
\end{equation}
where $a_1,\ldots,a_6 \in \Bbbk$ are parameters (see \cite{Lee}). 
A generic surface $Q_5 = \det A$ has an even set of  20 nodes given by the $4 \times 4$ minors of $A$. Hence, there exists a double cover
$\Phi \colon  F \rightarrow Q_5$ branched over these nodes.  Then $\Phi$ is the canonical map of $F$, and the canonical ring $R = R(F)$ is generated by $u_1,\ldots,u_4 \in R_1^{+}$, $v_1,\ldots,v_5 \in  R_2^{-}$ with the following relations 
\begin{equation}
\begin{aligned}
& \sum_{j}a_{ij}v_j, \textup{  (5 relations of degree 3)} \\
& v_jv_k - B_{jk}, \textup{  (15 relations of degree 4)}
\end{aligned}
\end{equation}
where $B_{jk}$ is the entry in row $j$ and column $k$ of the adjoint matrix of $A$ (see  \cite{CataneseBabbage}, Theorem 3.5). 
\begin{theorem}[\cite{Barlow}, Theorem 2.5 and subsequent Corollary]
	There exists an action of $D_5 = \langle \beta, 
	\alpha \rangle$ on $F$ such that $\beta$ acts freely on $F$ and $\alpha$ with a finite fixed locus. The  corresponding quotient is a surface  $B$ with four double points whose resolution is a minimal surface of general type with $K^2 = 1$, $p_g=0$ and $\pi_1 = \{1\}$.
\end{theorem}
The element $\alpha$ acts on $F$ via an induced action of $\tau$ on $F$ twisted by the canonical involution $\iota\colon F \rightarrow F$.  
The $D_5=\langle \beta,\alpha\rangle$-action on $F$ is given by 
\begin{alignat*}{3}
& \beta(u_i) = \xi ^i u_i,   \qquad &&\beta(v_i) = \xi^{-i} v_i, \\
& \alpha(u_i) = u_{-i} , \qquad  && \alpha(v_i) = -v_{-i} . 
\end{alignat*}
using indices in $\ZZ/5\ZZ$, see
\cite{BoehningBothmeretAl}, Remark 4.2.
\begin{example}
	The special surface constructed in \cite{Barlow} corresponds to a symmetric quintic as in \eqref{equ_closedemb} with parameters 
	\[ a_1 = a_2 = a_4 = a_5 = 1, \ a_3 = a_6 = -4,\] 
	see	\cite{BoehningBothmeretAl}, Remark 2.1.
	We take these parameters and construct with the help of \Mac2 the canonical ring $R(X) = R(F)^{D_5}$, the canonical model $\canmod$, a standard resolution of $R(X)$ as an $S$-module and the corresponding line $\ell$ in $Q$.
\end{example}

\begin{proposition}\label{prop_barlow}
	The Barlow surface is deformation equivalent to a general member of our 8-dimensional  family of torsion-free numerical Godeaux surfaces. In particular, all members of our dominant family are simply connected. 
\end{proposition}

\begin{proof} We verify computationally that over the Barlow line $\ell$ the solution space in the second step  is a 4-dimensional linear space, whence, as in the case of a line in $Q$ intersecting no homology loci, we obtain a $\P^3$ of solutions.  
	Using our procedure \href{https://www.math.uni-sb.de/ag/schreyer/images/data/computeralgebra/M2/doc/Macaulay2/NumericalGodeaux/html/_normal__Bundle__Line__In__Q.html}{normalBundleLineInQ}, 
	we determine the normal sheaf ${\cal N}_{\ell|Q}$ which is a line bundle as $\ell$ does not meet the singular locus of $Q$ and obtain
	\[
	{\cal N}_{\ell|Q} \cong \sO_{\P^1}(1)^2 \oplus \sO_{\P^1}^4.
	\]
	Thus, $h^0({\cal N}_{\ell|Q} ) =8$, $h^1({\cal N}_{\ell|Q} ) =0$ and we can move the Barlow line to a line in $Q$ not meeting any of the homology loci. Thus the Barlow surface lies in the dominant component. \end{proof}

\section{Summary and outlook}
In this paper we found an 8-dimensional family of numerical Godeaux surfaces with trivial fundamental group and whose bicanonical system on the 
canonical model has 4 distinct base points. 
In a forthcoming paper, we identify lines in $Q$ leading to $\ZZ/3\ZZ$- and $\ZZ/5\ZZ$-Godeaux surfaces. These lines meet one or two of the special loci introduced in Section \ref{subsec C-complex}. 
In the forthcoming paper, we investigate experimentally surfaces over finite fields which arise by taking lines through a general point on one or two of these loci. 
Consequently, whether our approach will lead to a complete classification of (marked) numerical Godeaux surfaces,  depends on  whether we will be able to exclude the existence of special lines, which lead to smooth surfaces.
Numerical Godeaux surfaces where the  base locus of $|2 K_{\canmod}|$ is non-reduced
will be treated separately. Modifying the deformation set-up for a minimal free resolution of the canonical ring,  we obtain a  description of an 8-dimensional unirational locally complete family of Godeaux surfaces with $\Tors=\ZZ/2\ZZ$ or  $\Tors=\ZZ/4\ZZ$ using our method. 

Recently, Dias and Rito  proved that the $\ZZ/2\ZZ$-Godeaux surfaces form a single unirational 8-dimensional family (see \cite{DiasRito20}).

\bigskip

\vbox{\noindent Author Addresses:\par
\smallskip
\noindent{Frank-Olaf Schreyer}\par
\noindent{Mathematik und Informatik, Universit\"at des Saarlandes, Campus E2 4, 
D-66123 Saarbr\"ucken, Germany.}\par
\noindent{schreyer@math.uni-sb.de}\par
\smallskip

\noindent{Isabel Stenger}\par
\noindent{Mathematik und Informatik, Universit\"at des Saarlandes, Campus E2 4, 
D-66123 Saarbr\"ucken, Germany.}\par
\noindent{stenger@math.uni-sb.de}\par

}

\end{document}